\documentclass[preprint,12pt]{elsarticle}
\usepackage{mathtools,amsthm,amssymb}                   
\usepackage{hyperref}
\usepackage{theoremref}
\usepackage{physics}
\usepackage{enumitem}
\usepackage{float}
\usepackage[dvipsnames,svgnames,table]{xcolor}
\usepackage{changepage}

\usepackage{pgfplots}
\pgfplotsset{compat=1.18}

\usepgfplotslibrary{external}
\everymath{\displaystyle}
\numberwithin{equation}{section}

\DeclareMathOperator{\conv}{\mathsf{conv}}
\DeclareMathOperator{\coni}{\mathsf{coni}}
\DeclareMathOperator{\cl}{\mathsf{cl}}                           

\newcommand{\bracket}[1]{\left[#1 \right]}				
\DeclareMathOperator{\spec}{\mathsf{spec}}

\newtheorem{theorem}{Theorem}[section]
\newtheorem{lemma}[theorem]{Lemma}

\newtheorem{proposition}[theorem]{Proposition}

\newtheorem{conjecture}[theorem]{Conjecture}

\theoremstyle{definition}
\newtheorem{definition}[theorem]{Definition}
\newtheorem{example}[theorem]{Example}

\theoremstyle{remark}
\newtheorem{remark}[theorem]{Remark}

\journal{J.~Algebra}

\begin{document}

\begin{frontmatter}
    \title{Character tables are ideal Perron similarities}
    \author[add]{David Z.~Gershnik}
        \ead{dgersh@uw.edu}
    \author[add]{Alexander J.~Lewis}
        \ead{lewis002@uw.edu}
    \author[add]{Pietro Paparella\corref{cor1}}
        \ead{pietrop@uw.edu}
        \affiliation[add]{organization={Division of Engineering \& Mathematics, University of Washington Bothell},
                    addressline={18115 Campus Way NE}, 
                    city={Bothell},
                    postcode={98011-8246}, 
                    state={WA},
                    country={U.S.A}}
        \cortext[cor1]{Corresponding author.}

\begin{abstract}
    An invertible matrix is called a \emph{Perron similarity} if it diagonalizes an irreducible, nonnegative matrix. Each Perron similarity gives a nontrivial polyhedral cone, called the \emph{spectracone}, and polytope, called the \emph{spectratope}, of realizable spectra (thought of as vectors in complex Euclidean space). A Perron similarity is called \emph{ideal} if its spectratope coincides with the conical hulls of its rows. Identifying ideal Perron similarities is of great interest in the pursuit of the longstanding \emph{nonnegative inverse eigenvalue problem}.
    
    In this work, it is shown that the character table of a finite group is an ideal Perron similarity. In addition to expanding ideal Perron similarities to include a broad class of matrices, the results unify previous works into a single, theoretical framework.   

    It is demonstrated that the spectracone can be described by finitely-many group-theoretic inequalities. When the character table is real, we derive a group-theoretic formula for the volume of the projected Perron spectratope, which is a simplex. Finally, an implication for further research is given.
\end{abstract}

\begin{keyword}
    character \sep character table \sep group \sep nonnegative matrix \sep Perron similarity \sep polyhedral cone \sep polytope \sep representation

    \MSC[2020] 15A29 \sep 20C15 \sep 20C99 \sep 15B51 \sep 52B99
\end{keyword}

\end{frontmatter}
    
\section{Introduction}

The \emph{nonnegative inverse eigenvalue problem (NIEP)} is to determine which multisets (hereinafter, \emph{lists}) of complex numbers occur as the spectra of entry-wise nonnegative matrices. The NIEP has proven to be one of the most challenging and sought-after problems in matrix analysis and includes a number of subproblems that are equally as daunting \cite{jmpp2018}. 

If $A$ is an $n$-by-$n$ entrywise nonnegative matrix with spectrum $\Lambda = \{\lambda_1,\ldots, \lambda_n \}$, then the list $\Lambda$ is called \emph{realizable}, and $A$ is called a \emph{realizing matrix (for $\Lambda$)}. In such a case,    
\begin{align}
\rho \left(\Lambda\right) &\coloneqq \max_{1 \leqslant k \leqslant n} \left\{ \vert\lambda_k\vert \right\} \in \Lambda, \label{sprad} \\		
\overline{\Lambda} &\coloneqq \left\{ \overline{\lambda_1},\dots, \overline{\lambda_n} \right\} = \Lambda,              \label{selfconj} 	\\ 
s_k (\Lambda) &\coloneqq \sum_{i=1}^n \lambda_i^k = \trace{(A^k)} \geqslant 0,~\forall k \in \mathbb{N},                \label{trnn}		\\
\intertext{and}
\left[ s_k (\Lambda) \right]^\ell &\leqslant n^{\ell-1} s_{k\ell} (\Lambda), \forall k, \ell \in \mathbb{N} \notag           \label{JLL}  
\end{align} 
are the well-known necessary conditions.

If $S \in \mathsf{GL}_n(\mathbb{F})$, where $\mathbb{F} \in \{\mathbb{R},\mathbb{C}\}$, and $e$ denotes the all-ones vector of appropriate size, 
then the polyhedral cone 
\[ \mathcal{C}(S) := \{ x \in \mathbb{F}^n \mid M_x \coloneqq SD_x S^{-1} \geqslant 0 \} \supseteq \{ \alpha e \mid \alpha \geqslant 0 \}, \] 
where $D_x$ is the diagonal matrix whose $(k,k)$-entry is $x_k$, is called the \emph{(Perron) spectracone of $S$}, and the polytope
\[ \mathcal{P}(S) := \left\{ x \in \mathcal{C}(S) \mid M_x e = e \right\} \supseteq \{e\}, \]
is called the \emph{(Perron) spectratope of $S$}. If there is a vector $x \in \mathbb{F}^n$ such that the matrix $M_x$ is irreducible and nonnegative, then $S$ is called a \emph{Perron similarity}. The matrix $S$ is called \emph{ideal} if $\mathcal{C}(S) = \mathcal{C}_r(S)$, where $\mathcal{C}_r(S)$ denotes the conical hull of the rows of $S$. 

Identifying ideal Perron similarities is and has been of interest in the study of the NIEP \cite{jp2016,jp2017,jp2025}. In particular:
\begin{itemize}
    \item A matrix $H = [h_{ij}] \in \mathsf{M}_n(\mathbb{R})$ is called a \emph{Hadamard matrix} if $h_{ij} \in \{ \pm 1 \}$ and $H H^\top = nI_n$. A Hadamard matrix $H$ is called \emph{normalized} if $He_1 = e$ and $e_1^\top H = e^\top$. Johnson and Paparella \cite{jp2016} used the theory of \emph{association schemes} to show that the \emph{Walsh matrix}
    \begin{equation}
        \label{walsh}
        H_{2^n} \coloneqq
        \begin{bmatrix}
            1 & 1   \\
            1 & -1
        \end{bmatrix}^{\otimes n},\ n \in \mathbb{N}_0    
    \end{equation}
    is ideal. Although every normalized Hadamard matrix is a Perron similarity, not every such matrix is ideal \cite[Remark 4.31]{jp2025}.

    \item A matrix $H = [h_{ij}] \in \mathsf{M}_n(\mathbb{C})$ is called a \emph{complex Hadamard matrix} if $\vert h_{ij}\vert = 1$ and $HH^\ast = nI_n$. A complex Hadamard matrix $H$ is called \emph{dephased} if $He_1 = e$ and $e_1^\top H = e^\top$.

    In a more recent work, Johnson and Paparella \cite[Corollary 7.1]{jp2025} showed that the non-normalized discrete Fourier transform matrix 
         \begin{equation}
             \label{dft}
             F = F_n = [\omega_n^{(i-1)(j-1)}],\ \omega_n \coloneqq e^{\frac{2\pi}{n}\mathsf{i}},
         \end{equation}
    a dephased complex Hadamard matrix, is also ideal by appealing to the fact that $F_n$ is a Vandermonde matrix with respect to the zeros of the polynomial $t^n - 1$. Furthermore, any arbitrary Kronecker product of the form
        \begin{equation}
        \label{arbitrary}
        \left( \bigotimes_{j=1}^N F_{n_j} \right) \otimes H_{2^k},\ k \in \mathbb{N}_0,\ n_j \in \mathbb{N},\ 1 \leqslant j \leqslant N,
        \end{equation}
    is also ideal \cite[Corollary 7.17]{jp2025}. 
\end{itemize}

If $H$ is a dephased complex Hadamard matrix, then $H$ is a Perron similarity because
$$\frac{1}{n}H (e_1e_1^\top) H^\ast = \frac{1}{n}ee^\top > 0,$$
but a dephased complex Hadamard matrix need not be ideal: notice that there are values $\theta \in [0,\pi)$ such that the multisets
\[ \{ 1, \mathsf{i}e^{\mathsf{i}\theta}, -1, -\mathsf{i}e^{\mathsf{i}\theta} \} \]
and 
\[ \{ 1, -\mathsf{i}e^{\mathsf{i}\theta}, -1,\mathsf{i}e^{\mathsf{i}\theta} \} \]
fail the self-conjugacy condition \eqref{selfconj}. Consequently, there are values $\theta \in [0,\pi)$ such that the dephased complex Hadamard matrix
\[ F_4^{(1)}(\theta) \coloneqq 
\begin{bmatrix}
    1 & 1 & 1 & 1 \\
    1 & \mathsf{i}e^{\mathsf{i}\theta} & -1 & -\mathsf{i}e^{\mathsf{i}\theta} \\
    1 & -1 & 1 & -1 \\
    1 & -\mathsf{i}e^{\mathsf{i}\theta} & -1 & \mathsf{i}e^{\mathsf{i}\theta}
\end{bmatrix},\ \theta \in [0,\pi).
\]
is not ideal. 

It is thus natural to ask if the matrices given by equations \eqref{walsh}--\eqref{arbitrary} belong to a larger class of ideal Perron similarities. Interestingly, these matrices are \emph{character tables}, which leads naturally to the question: is every character table an ideal Perron similarity?

We utilize character theory to affirmatively answer the question posed above. This not only expands ideal Perron similarities to a broad class of matrices but also unifies results that appeared previously in the literature \cite{dppt2022, jp2016,jp2017,jp2025} within a single theoretical framework.

Constructive results are highly sought-after in the NIEP. Indeed, according to Chu \cite{c1998}: 
\begin{adjustwidth}{2.5em}{0pt}
Very few of these theoretical results are ready for implementation to actually compute [the realizing] matrix. The most constructive result we have seen is the sufficient condition studied by Soules [175]. But the condition there is still limited because the construction depends on the specification of the Perron vector---in particular, the components of the Perron eigenvector need to satisfy certain inequalities in order for the construction to work. \cite[p.~18]{c1998}. 
\end{adjustwidth}
Unfortunately, what Chu wrote in the late 1990s still rings true---for example, a constructive version of a result due to Fiedler that every \emph{Sule\u{\i}manova spectrum} is symmetrically realizable is only known for low dimensions and Hadamard orders \cite{jp2016}.

The results contained here are constructive: every character table yields a polytope of realizable spectra that is closed with respect to the Hadamard product. Of particular interest are the extreme points of the polytope corresponding to the character table of a finite Abelian group as they are located on the boundary of the set of realizable spectra. 

\section{Background \& Notation}

For ease of notation, $\mathbb{N}$ denotes the set of positive integers and $\mathbb{N}_0 \coloneqq \mathbb{N} \cup \{ 0 \}$. If $n \in \mathbb{N}$, then $\bracket{n} \coloneqq \{ k \in \mathbb{N} \mid 1 \leqslant k \leqslant n \}$.

The set of $m$-by-$n$ matrices with entries from a field $\mathbb{F}$ is denoted by $\mathsf{M}_{m \times n}(\mathbb{F})$. If $m = n$, then $\mathsf{M}_{n \times n}(\mathbb{F})$ is abbreviated to $\mathsf{M}_n(\mathbb{F})$. The set of nonsingular matrices in $\mathsf{M}_n(\mathbb{F})$ is denoted by $\mathsf{GL}_{n}\left(\mathbb{F}\right)$. 

If $x \in \mathbb{F}^n$, then $x_k$ or $[x]_k$ denotes the $k\textsuperscript{th}$-entry of $x$ and $D_x \in \mathsf{M}_n$ denotes the diagonal matrix whose $(k,k)$-entry is $x_k$. A vector $x \in \mathbb{F}^n$ is called \emph{totally nonzero} if $x_k \ne 0,\ \forall k \in \bracket{n}$. If $x \in \mathbb{F}^n$ is totally nonzero, then $(D_x)^{-1} = D_{x^{-1}}$, where $x^{-1}$ denotes the entrywise inverse of $x$ (i.e., the inverse with respect to the Hadamard product).

By the mechanics of matrix multiplication, if $S \in \mathsf{GL}_n(\mathbb{F})$ and $t_{ij}$ denotes the $(i,j)$-entry of $S^{-1}$, then
\begin{equation}
   \left[ S D_x S^{-1} \right]_{ij} = \sum_{k=1}^n s_{ik} (x_{k} t_{kj}) = \sum_{k=1}^n (s_{ik} t_{kj}) x_{k}. \label{(i,j)-entry}
\end{equation}

Denote by $I$, $e$, $e_k$, and $0$ the identity matrix, the all-ones vector, the $k\textsuperscript{th}$ canonical basis vector, and the zero vector, respectively. The size of each aforementioned object is implied by its context.

If $A \in \mathsf{M}_{m \times n}(\mathbb{F})$, then:

    \begin{itemize}
        \item $a_{ij}$, $a_{i,j}$, or $[A]_{ij}$ denotes the $(i,j)$-entry of $A$;
        \item $A^\top$ denotes the \emph{transpose of $A$};
        \item $\overline{A} = [\overline{a_{ij}}]$ denotes the (entrywise) conjugate of $A$; 
        \item $A^\ast \coloneqq  \overline{A^\top} = \overline{A}^\top$ denotes the conjugate-transpose of $A$;  
        \item $r_i(A) \coloneqq A^\top e_i$ denotes the $i\textsuperscript{th}$-row of $A$ as a column vector and when the context is clear, $r_i(A)$ is abbreviated to $r_i$; and  
        \item when $m=n$, $\spec A = \spec(A)$ denotes the \emph{spectrum of $A$} and $\rho = \rho(A)$ denotes the \emph{spectral radius of $A$}.
    \end{itemize}

If $\varnothing \subset X \subseteq \mathbb{F}^n$, then the \emph{conical hull of $X$}, denoted by $\coni X = \coni(X)$, is defined by 
\[ \coni {X} = \left\{ \sum_{k=1}^m \alpha_k x_k \in \mathbb{F}^n \mid m \in \mathbb{N},~x_k \in X,~\alpha_k \geqslant 0 \right\} \]
i.e., $\coni X$ consists of all \emph{conical combinations}. Similarly, the \emph{convex hull of $X$}, denoted by $\conv{X} = \conv(X)$ is defined by  
\[ \conv {X} = \left\{ \sum_{k=1}^m \alpha_k x_k \in \mathbb{F}^n \mid m \in \mathbb{N},~x_k \in X,\ \sum_{k=1}^m \alpha_k = 1,\ \alpha_k \geqslant 0 \right\}. \]
The conical hull or convex hull of a finite list $\{x_1,\ldots,x_n\}$ is abbreviated to $\coni(x_1,\ldots,x_n)$ or $\conv(x_1,\ldots,x_n)$, respectively.

\subsection{Characters \& Representations}

Hereinafter, $G = (G,e_G, \cdot)$ is a finite group. For ease of notation, $a\cdot b$ is abbreviated to $ab$. If $g \in G$, then $\cl(g) \coloneqq \{aga^{-1}\mid a \in G\}$ denotes the \emph{conjugacy class of $g$} and $C_G (g) \coloneqq \{ a \in G \mid ag = ga \}$ denotes the \emph{centralizer of $g$ (in $G$)}. Because $\{ \cl(g) \mid g \in G \}$ forms a partition of $G$, we write $x \sim y$ whenever $x,y \in \cl(g)$.

If $\rho:G \longrightarrow \textsf{GL}_n(\mathbb{C})$ is a homomorphism, then $\rho$ is called a \emph{(matrix) representation (of $G$)}. As $\rho$ is a homomorphism, $\rho(e_G) = I_n$ and $\rho(g^{-1}) = \rho(g)^{-1}$. 

If $\rho$ and $\sigma$ are representations, then:

\begin{itemize}
    \item the \emph{conjugate transpose of $\rho$}, denoted by $\rho^\ast$, is the representation defined by $\rho^\ast(g) = \rho(g)^\ast,\ \forall g \in G$;
    \item the \emph{direct sum of $\rho$ and $\sigma$}, denoted by $\rho\oplus\sigma$, is the representation defined by $(\rho \oplus \sigma)(g) = \rho (g) \oplus \sigma (g),\ \forall g \in G$;  
    \item the \emph{Kronecker} or \emph{tensor product} of $\rho$ and $\sigma$, denoted by $\rho \otimes \sigma$, is the representation defined by $(\rho \otimes \sigma)(g) = \rho (g) \otimes \sigma (g),\ \forall g \in G$; and
    \item $\rho$ and $\sigma$ are \emph{isomorphic} or \emph{similar} if there is an invertible matrix $S$ such that $\rho(g) = S\sigma(g) S^{-1},\ \forall g \in G$.
\end{itemize}

For a representation $\rho$, the function $\chi_\rho : G \longrightarrow \mathbb{C}$, defined by 
\[ \chi_\rho(g) = \text{tr}(\rho(g)),\ \forall g \in G, \] 
is called the \emph{character of $\rho$}. The \emph{degree} or \emph{dimension of $\rho$}, denoted by $\dim \rho = \dim (\rho)$, is the quantity $n = \trace{I_n} = \chi_\rho (e_G)$. The \emph{degree} or \emph{dimension of $\chi_\rho$} is the same quantity.

The following properties are well known (see, e.g., Fulton and Harris \cite[p.~13]{fh1991} or Serre \cite[Propositions 1 and 2]{s1977}):
\begin{align}
    \chi_{\rho \oplus \sigma} (g) &= \chi_\rho (g) + \chi_\sigma (g)        \notag             \\
    \chi_{\rho \otimes \sigma} (g) &= \chi_\rho (g) \cdot \chi_\sigma (g)   \label{tensorprod} \\     
    \chi_{\rho^\ast} (g) &= \overline{\chi_\rho (g)} = \chi_\rho(g^{-1})    \label{conjtrans}  \\
    \chi_\rho (g) &= \chi_\rho (h),\ g \sim h                               \label{classfun}
\end{align}
Equation \ref{classfun} ensures that $\chi_\rho$ is a \emph{class function}.

A subspace $W$ of $\mathbb{C}^n$ is called \emph{$G$-invariant} or a \emph{$G$-invariant subspace} if 
$$\rho(g) w \in W,\ \forall g \in G,\ \forall w \in W.$$ 
Obviously, the trivial subspaces $\{ 0 \}$ and $\mathbb{C}^n$ are $G$-invariant; if $\mathbb{C}^n$ has a proper $G$-invariant subspace, then $\rho$ is called \emph{reducible} or a \emph{reducible representation}. If $\mathbb{C}^n$ has no proper $G$-invariant subspace, then $\rho$ is called \emph{irreducible} or an \emph{irreducible representation}. A character of an irreducible representation is called an \emph{irreducible character}.  

If $\rho$ is a {reducible representation}, then there is an invertible matrix $S$ such that 
\[ S^{-1} \rho(g) S = 
\begin{bmatrix}
\rho_1(g) & \sigma(g) \\
0 & \rho_2(g)
\end{bmatrix},\]
where $\rho_1 = \left.\rho\right|_W$, $\rho_2$ is the representation of $G$ on $\mathbb{C}^n/ W$, and $\sigma(g)$ is a rectangular matrix \cite[p.~848]{df2004}.

The representation $g \in G \longmapsto I_n\in \mathsf{GL_n(\mathbb{C})}$ is called the \emph{($n$-dimensional) trivial representation} and is irreducible when and only when $n=1$.  

If $\rho$ and $\sigma$ are representations, then  
\begin{equation*}\label{char_inner_prod}
    \langle \chi_\rho,\chi_\sigma \rangle 
    \coloneqq \frac{1}{|G|} \sum_{g \in G} \overline{\chi_\rho(g)} \chi_\sigma(g) 
    = \frac{1}{|G|}\sum_{k=1}^n|\cl(g_k)| \left( \overline{\chi_\rho(g_k)} \chi_\sigma(g_k) \right)
\end{equation*}
defines an inner product on characters (see, e.g., Artin \cite[pp.~300]{a1991}).

The number of distinct irreducible representations (up to isomorphism) of a finite group equals the number of its distinct conjugacy classes (see, e.g., Artin \cite[Theorem 10.4.6(b)]{a1991}; Fulton and Harris \cite[Corollary 2.13]{fh1991}; or Serre \cite[Theorem 7]{s1977}). The following well-known theorem (see, e.g., Artin \cite[Theorem 10.4.6(a) and Corollary 10.4.8]{a1991} or Dummit and Foote \cite[Theorem 15]{df2004}) will be crucial in what follows.

\begin{theorem}
    \thlabel{charlinearcombo}
    If $\rho_1, \ldots, \rho_n$ are the distinct irreducible representations of $G$ and $\chi_1, \ldots, \chi_n$ are their respective characters, then 
    \begin{equation*}
        \langle \chi_i,\chi_j \rangle = \delta_{ij}. \label{sor_rows}
    \end{equation*}
    If $\rho$ is a representation, then 
    \begin{equation}
        \label{charnnlc}
        \chi_{\rho} (g) = \sum_{k = 1}^n \langle \chi_{\rho}, \chi_k \rangle \chi_k (g),\ \forall g \in G.
    \end{equation}
\end{theorem}

Every character is, in fact, a nonnegative integral linear combination of the irreducible characters \cite[p.~868]{df2004}. 

\subsection{Character Tables}

Suppose that $\rho_1, \ldots, \rho_n$ are, up to isomorphism, the distinct irreducible representations of $G$ and $\cl (g_1), \ldots, \cl (g_n)$ are the distinct conjugacy classes of $G$. For ease of notation, $\chi_{\rho_k}$ is abbreviated to $\chi_k$. The $n$-by-$n$ matrix $Q$ 
\[ Q = 
\begin{bmatrix}
    \chi_1(g_1) & \cdots & \chi_1(g_n) \\
    \vdots      & \ddots & \vdots       \\
    \chi_n(g_1) & \cdots & \chi_n(g_n)
\end{bmatrix}, \]
i.e., $q_{ij} \coloneqq \chi_i(g_j)$, is called the \emph{character table of $G$}. The class function property \eqref{classfun} ensures that the matrix $Q$ is unique up to permutation similarity.  

Let $\rho_1: G_1 \longrightarrow \mathsf{GL}_{n_1}(\mathbb{C})$ and $\rho_2: G_2 \longrightarrow \mathsf{GL}_{n_2}(\mathbb{C})$ be representations of finite groups $G_1$ and $G_2$, respectively. The Kronecker product of matrices defines a representation of $G_1 \times G_2$, also called the tensor product, via 
\[ (\rho_1 \otimes \rho_2)(g_1,g_2) = \rho_1 (g_1) \otimes \rho_2 (g_2),\ \forall (g_1,g_2) \in G_1 \times G_2.\]
If, in addition, $\rho_1$ and $\rho_2$ are irreducible, then $\rho_1 \otimes \rho_2$ is an irreducible representation of $G_1 \times G_2$ \cite[Theorem 10(i)]{s1977}. Conversely, every irreducible representation of $G_1 \times G_2$ is isomorphic to a tensor product $\rho_1 \otimes \rho_2$, where $\rho_1$ and $\rho_2$ are irreducible representations of $G_1$ and $G_2$, respectively \cite[Theorem 10(ii)]{s1977}. Hence, if $Q_1$ and $Q_2$ are character tables of $G_1$ and $G_2$, respectively, then $Q_1 \otimes Q_2$ is the character table of $G_1 \times G_2$. This result generalizes as follows: if $Q_1,\ldots,Q_m$ are character tables of $G_1,\ldots,G_m$, respectively, then $\bigotimes_{k = 1}^m Q_k$ is the character table of $\prod_{k = 1}^m G_k$.  

\subsection{Nonnegative \& Stochastic Matrices}

If $A \in \mathsf{M}_{n}(\mathbb{C})$, $n \geqslant 2$, then $A$ is called \emph{reducible} if there is a permutation matrix $P$ such that
\begin{align*}
P^\top A P =
\begin{bmatrix}
A_{11} & A_{12} \\
0 & A_{22}
\end{bmatrix},
\end{align*}
where $A_{11}$ and $A_{22}$ are non-empty square matrices. If $A$ is not reducible, then A is called \emph{irreducible}. Clearly, entrywise positive matrices are irreducible.  

If $x \in \mathbb{R}^n$ and $x_i \geqslant 0,\ \forall i \in \bracket{n}$ ($x_i > 0,\ \forall i \in \bracket{n}$), then is called \emph{nonnegative} (respectively, \emph{positive}), and we write $x \geqslant 0$ (respectively, $x > 0$). A similar definition applies to real matrices. If $x \in \mathbb{R}^n$ ($A \in \mathsf{M}_n(\mathbb{R})$), then $x$ (respectively, $A$) is viewed as a complex vector (respectively, matrix) via the map $x \longmapsto x+ 0\mathsf{i}$ (respectively, $A \longmapsto A + 0\mathsf{i}$).

If $x \in \mathbb{C}^n$, $\Re x_i \geqslant 0,\ \forall i \in \bracket{n}$ ($\Re x_i > 0,\ \forall i \in \bracket{n}$), and $\Im x_i = 0,\ \forall i \in \bracket{n}$, then $x$ is called \emph{nonnegative} (respectively, \emph{positive}), and we write $x \geqslant 0$ (respectively, $x > 0$). A similar definition applies to complex matrices.  

If $A \geqslant 0$ and $$\sum_{j=1}^n a_{ij} = 1,~\forall~i \in \bracket{n},$$ then $A$ is called \emph{(row) stochastic}. If $A \geqslant 0$, then $A$ is stochastic if and only if $Ae = e$. Furthermore, if $A$ is stochastic, then $1 \in \spec(A)$ and $\rho(A) = 1$. It is known that the NIEP and the stochastic NIEP are equivalent (see, e.g., Johnson \cite[p.~114]{j1981}).

\section{Spectral Polyhedra \& Perron Similarities}

\begin{definition}
    \thlabel{spectratope}
    If $S \in \mathsf{GL}_{n}(\mathbb{C})$, then: 
    \begin{enumerate}
        [label=(\roman*)]
        \item $\mathcal{C}(S) \coloneqq \{ x \in \mathbb{C}^n \mid M_x = M_x(S) \coloneqq S D_x S^{-1} \geqslant 0 \}$ is called the \emph{(Perron) spectracone} of $S$;
        \item $\mathcal{P}(S) \coloneqq \{ x \in \mathcal{C}(S) \mid M_x e = e \}$ is called the \emph{(Perron) spectratope} of $S$; and
        \item $\mathcal{A}(S) \coloneqq \{ M_x \in \mathsf{M}_n(\mathbb{R}) \mid x \in \mathcal{C}(S) \}$. 
    \end{enumerate}
\end{definition}

Since $M_e = S D_e S^{-1} = S I S^{-1} = I \geqslant 0$ and $I$ is stochastic, it follows that $e \in \mathcal{P}(S) \subset \mathcal{C}(S)$ and all three sets are nonempty. If $\mathcal{C}(S) = \coni(e)$, $\mathcal{P}(S) = \{e\}$, or $\mathcal{A}(S) = \coni(I_n)$, then $\mathcal{C}(S)$, $\mathcal{P}(S)$, and $\mathcal{A}(S)$ are called \emph{trivial}; otherwise, they are called \emph{nontrivial}.

It is known that $\mathcal{C}(S)$ is a polyhedral cone and $\mathcal{P}(S)$ is a polytope and both sets are closed with respect to the Hadamard product \cite[Theorems 4.4 and 4.6]{jp2025}. The set $\mathcal{A}(S)$ is a cone that is closed with respect to matrix multiplication \cite[Theorem 4.4(iii)]{jp2025}.

Recall from the introduction that an invertible matrix $S$ is called \emph{ideal} if $\mathcal{C}(S) = \mathcal{C}_r(S)$, where $\mathcal{C}_r(S)$ denotes the conical hull of the rows of $S$. It is known that $S$ is ideal if and only if $e \in \mathcal{C}_r(S)$ and $r_i \in \mathcal{C}(S),\ \forall i \in \bracket{n}$, i.e., every row, viewed as a column vector, belongs to its spectracone \cite[Theorem 4.21]{jp2025}.

Building on a definition that previously appeared for real matrices \cite[Definition 2.3]{jp2017}, we now extend the concept of \emph{row Hadamard conic} to complex matrices.

\begin{definition}
    \thlabel{rhc}
        If $S \in \mathsf{M}_{m \times n}({\mathbb{C}})$, then $S$ is called \emph{row Hadamard conic (RHC)} if \( r_i \circ r_j \in \mathcal{C}_r(S),~\forall i,j \in \bracket{m} \).
\end{definition}

\begin{proposition}
    [{\cite[Proposition 4.16]{jp2025}}]
    \thlabel{simple}
        If $S \in \mathsf{GL}_n(\mathbb{C})$, then $x^\top S^{-1} \geqslant 0$ if and only if $x \in \mathcal{C}_r(S)$. 
\end{proposition}

\begin{proposition}
    \thlabel{rhcrows}
        If $S \in \mathsf{GL}_n(\mathbb{C})$, then $r_i \in \mathcal{C}(S)$ if and only if $r_i \circ r_j \in \mathcal{C}_r(S),\ \forall j \in \bracket{n}$.
\end{proposition}

\begin{proof}
Because the $j$\textsuperscript{th}-row of $S  D_{r_i}$ is $r_j \circ r_i$, it follows that 
\begin{align*}
r_i \in \mathcal{C}(S) 
& \Longleftrightarrow S D_{r_i} S^{-1} \geqslant0 											\\ 
& \Longleftrightarrow (S D_{r_i}) S^{-1} \geqslant0 										\\
& \Longleftrightarrow \left( r_j \circ r_i \right) S^{-1} \geqslant0,~\forall j \in \bracket{n} 					\\
& \Longleftrightarrow r_j \circ r_i = r_i \circ r_j \in \mathcal{C}_r(S),~\forall j \in \bracket{n} \tag*{[\thref{simple}]}
\end{align*}
as desired.
\end{proof}

The following characterization is an immediate result of \thref{rhcrows}.

\begin{theorem}
\thlabel{RHC ideal}
    If $S \in \mathsf{GL}_n (\mathbb{C})$, then $S$ is ideal if and only if $e \in \mathcal{C}_r(S)$ and $S$ is RHC.
\end{theorem}

If $S \in \mathsf{GL}_n(\mathbb{C})$, then $S$ is called a \emph{Perron similarity} if there is a diagonal matrix $D$ such that $A = SDS^{-1}$ is irreducible and non-negative. If $S \in \mathsf{GL}_n(\mathbb{C})$, then $S$ is a Perron similarity if and only if there is a unique positive integer $k \in \bracket{n}$ such that $S e_k = \alpha x$ and $e_k^\top S^{-1} = \beta y^\top$, where $\alpha$ and $\beta$ are complex numbers such that $\alpha \beta > 0$, and $x$ and $y$ are positive vectors \cite[Theorem 4.14]{jp2025}.

\begin{lemma}
\thlabel{epairconstruct}
    Supoose that $S \in \mathsf{GL}_n (\mathbb{C})$, $\lambda \in \mathbb{C}^n$, and $A \coloneqq S D_\lambda S^{-1}$. If $\exists k \in [n]$ such that $v \coloneqq S e_k$ is totally nonzero, then $(\lambda_k, e)$ is an eigenpair of $D_{v^{-1}} A D_v$.

    \begin{proof}
        Notice that 
        \begin{align*}
            (D_{v^{-1}} A D_v) e &= D_{v^{-1}} ((S D_\lambda S^{-1} ) v)\\
            &= D_{v^{-1}} ((S D_\lambda S^{-1} ) S e_k)\\
            &= D_{v^{-1}} ((S D_\lambda) e_k)\\
            &= D_{v^{-1}} (\lambda_k v)\\
            &= \lambda_k e,
        \end{align*}
        and the result is established.
    \end{proof}
\end{lemma}

If $S \in \mathsf{GL}_n(\mathbb{C})$ and $v > 0$, then $\mathcal{C}(S) = \mathcal{C}(D_vS)$ \cite[Theorem 4.10(iii)]{jp2025}. However, matrix transformations with respect to spectratopes remain unexplored. The following result is novel.

\begin{theorem}
\thlabel{P(S)contain}
    If $S \in \mathsf{GL}_n (\mathbb{C})$ is a Perron similarity, then there exists $k \in [n]$ such that $\mathcal{P}(S) \subseteq \mathcal{P}(D_{v^{-1}} S)$, where $v \coloneqq S e_k$.
\end{theorem}

\begin{proof}
    As $S$ is a Perron similarity, there is a positive integer $k \in [n]$, a nonzero complex number $\alpha$, and a positive vector $x$ such that $v \coloneqq S e_k = \alpha x$ is totally nonzero.
    
    If $z \in \mathcal{P}(S)$, then $M_z(S)$ is a stochastic matrix. Since $(z_k,v)$ is an eigenpair, it follows that $(z_k, x)$ is an eigenpair. As $M_z(S)$ is a nonnegative matrix having a positive eigenvector, it is necessarily the case that $z_k = \rho(M_z(S))$ \cite[Corollary 8.1.30]{hj2013}. As $M_z(S)$ is stochastic, $z_k = 1$. Finally, \thref{epairconstruct} ensures that $(z_k,e) = (1,e)$ is an eigenpair of $M_z(D_{v^{-1}}S)$.     
\end{proof}

\begin{example}
    The inclusion in \thref{P(S)contain} can be strict: if $S = \begin{bmatrix}
        1 & 1 \\
        2 & -2
    \end{bmatrix}$, then the matrix 
    \[ S D_x S^{-1} = 
    \begin{bmatrix}
        \frac{x_1 + x_2}{2} & \frac{x_1 - x_2}{4} \\ x_1 - x_2 & \frac{x_1 + x_2}{2}       
    \end{bmatrix}\]
    is stochastic if and only if $x_1 = x_2 = 1$, i.e., $\mathcal{P}(S)$ is trivial. However, the matrix $H_2 = \begin{bmatrix}
        1 & 1 \\
        1 & -1
    \end{bmatrix}$
    is ideal.
\end{example}

\section{Main Results}

In what follows, $Q$ is the character table of a finite group $G$. Since character tables are unique up to permutation similarity, it is assumed that $g_1 = e_G$ and $\rho_1$ is the one-dimensional trivial representation.

\begin{theorem}
    The matrix $Q$ is a Perron similarity.
\end{theorem}

\begin{proof}
Since 
    \begin{equation*}
        \sum_{k=1}^n \overline{\chi_k(g)} \chi_k(h) =
        \begin{cases}
            \vert C_G (g) \vert, & g \sim h \\
            0, & g \nsim h
        \end{cases}
    \end{equation*}
(see, e.g., Dummit and Foote \cite[Theorem 16]{df2004}), it follows that
    \begin{equation*}
    \langle Qe_j, Q e_i \rangle =
    (Q e_i)^\ast Q e_j = 
    \sum_{k=1}^n \overline{\chi_k(g_i)} \chi_k(g_j) = 
    \begin{cases} \vert C_G (g_i) \vert,    & i = j \\ 
    0,                                      & i \neq j
    \end{cases}
    \end{equation*}
i.e., the columns of $Q$ are orthogonal. Moreover, the matrix $D \coloneqq Q^* Q$ is an invertible diagonal matrix such that
    \begin{equation*}
    d_{ij} = 
    \begin{cases}
        \vert C_G (g_i) \vert, & i = j \\
        0, & i \neq j.
    \end{cases}
    \end{equation*}
Thus, $Q^{-1} = D^{-1} Q^\ast$ and 
    \begin{equation}
    \label{char_table_inverese_entrywise}
        \left[ Q^{-1} \right]_{ij} = \frac{\overline{q_{ji}}}{\vert C_G(g_i) \vert} 
    \end{equation}
Since $g_1 = e_G$, we have 
    \begin{equation*}
    q_{i,1} = \chi_i (e_G) = \trace{(I_{\dim(\rho_i)})} = \dim(\rho_i) > 0.  \label{firstcolumn}    
    \end{equation*}
Via Equation \ref{(i,j)-entry}, 
    \begin{align*}
        \left[ M_{e_1} (Q)\right]_{ij}
        &= \left[ Q D_{e_1} Q^{-1} \right]_{ij}                                     \\
        &= \sum_{k=1}^n \left( q_{ik} \left[Q^{-1} \right]_{kj} \right) [e_1]_k     \\
        &= \frac{q_{i,1} \overline{q_{j,1}}}{\vert C_G(e_G) \vert}                  \\
        &= \frac{\dim (\rho_i) \dim (\rho_j)} {\vert G \vert} > 0.
    \end{align*}
Thus, $M_{e_1} (Q) > 0$.
\end{proof}

\begin{theorem}
\thlabel{extremerayQ}
    The matrix $Q$ is ideal.
\end{theorem}

    \begin{proof}
    Since $q_{1,j} = \chi_1 (g_j) = \trace{\left(\rho_1(g_j) \right)} = 1$, it follows that the first row of $Q$ is the all-ones vector and $e \in \mathcal{C}_r(Q)$. 

    The result is clear because the $k$th entry of $r_i^\top \circ r_j^\top$ is 
    $$\chi_i (g_k) \cdot \chi_j (g_k) = \chi_{\rho_i \otimes \rho_j} (g_k) = \sum_{\ell = 1}^n \langle \chi_{\rho_i \otimes \rho_j}, \chi_\ell \rangle \chi_\ell (g_k),$$
    which demonstrates that $Q$ is RHC and thus ideal by \thref{RHC ideal}.
    \end{proof}

Whereas \thref{extremerayQ} yields a description of $\mathcal{C}(Q)$ as the conical hull of a set of linearly independent vectors, the subsequent result describes $\mathcal{C}(Q)$ in terms of linear inequalities and demonstrates that not all of the $n^2$ inequalities of the matrix equation $M_x(Q) \geqslant 0$ are required to specify the cone. The linear inequalities are novel sufficient conditions for realizability.

\begin{theorem}
\thlabel{ineqredux}
    If $x \in \mathbb{C}^n$, then $M_x (Q) \geqslant 0$ if and only if 
    \begin{equation} \label{firstcolumn2}
        \sum_{k = 1}^n |\cl (g_k)| \chi_i (g_k) x_k \geqslant 0, \hspace{10pt} \forall i \in [n].
    \end{equation}
\end{theorem}

\begin{proof}
    First, recall that 
    \begin{equation}
        \label{g=Cg cl g}
        \vert G \vert = \vert C_G(g)\vert \vert\cl(g)\vert,\ \forall g \in G
    \end{equation} 
    (see, e.g., Artin \cite[p.~196]{a1991}). By equations \ref{(i,j)-entry}, \ref{char_table_inverese_entrywise}, and \ref{g=Cg cl g}, 
    \begin{align*}
        \left[ M_{x} (Q)\right]_{ij}
        &= \sum_{k=1}^n \left( q_{ik} \left[Q^{-1} \right]_{kj} \right) x_k                 \\
        &= \sum_{k=1}^n \frac{\chi_i(g_k)\overline{\chi_j (g_k)}}{\vert C_G(g_k) \vert} x_k \\
        &= \frac{1}{\vert G \vert} \sum_{k=1}^n \vert \cl(g_k) \vert \chi_i(g_k)\overline{\chi_j (g_k)} x_k.
    \end{align*}  
    To demonstrate the necessity of Condition \ref{firstcolumn2}, notice that if $M_{x} (Q) \geqslant 0$ and $i \in \bracket{n}$, then   
    $$0 \leqslant [M_x (Q)]_{i 1} = \frac{1}{|G|} \sum_{k = 1}^n \vert\cl(g_k)\vert \chi_i (g_k) \overline{\chi_1 (g_k)} x_k = \frac{1}{|G|} \sum_{k = 1}^n |\cl(g_k)| \chi_i (g_k) x_k.$$

    Conversely, if Condition \ref{firstcolumn2} holds and $i, j \in [n]$, then following properties of representations and \thref{charlinearcombo}
    \begin{align*}
        [M_x (Q)]_{i j} 
        &= \frac{1}{|G|} \sum_{k = 1}^n |\cl(g_k)| \chi_i (g_k) \overline{\chi_j (g_k)} x_k                 \\
        &= \frac{1}{|G|} \sum_{k = 1}^n |\cl(g_k)| \chi_{\rho_i \otimes \rho_j^*} (g_k) x_k                                                                         \\
        &= \frac{1}{|G|} \sum_{k = 1}^n \left( \sum_{\ell = 1}^n |\cl(g_k)| \langle \chi_{\rho_i \otimes \rho_j^*}, \chi_\ell \rangle \chi_\ell (g_k) x_k \right) \\
        &= \frac{1}{|G|} \sum_{\ell = 1}^n \left( \sum_{k = 1}^n |\cl(g_k)| \langle \chi_{\rho_i \otimes \rho_j^*}, \chi_\ell \rangle \chi_\ell (g_k) x_k \right)   \\
        &= \frac{1}{|G|} \sum_{\ell = 1}^n \langle \chi_{\rho_i \otimes \rho_j^*}, \chi_\ell \rangle \left( \sum_{k = 1}^n |\cl(g_k)| \chi_\ell (g_k) x_k \right) \geqslant 0, 
    \end{align*}
    which establishes the sufficiency of Condition \ref{firstcolumn2}.
\end{proof}

\begin{remark}
    If $\Im Q \ne 0$, then $M_x$ is normal since
    \begin{align*}
        M_x^\ast M_x 
        &= (Q D_x D^{-1}Q^\ast)^\ast(Q D_x D^{-1}Q^\ast)                                                \\
        &= (Q \overline{D^{-1}} D_{\overline{x}} Q^\ast)(Q D_x D^{-1}Q^\ast)                            \\
        &= Q (\overline{D^{-1}} D_{\overline{x}} D D_x D^{-1}) Q^\ast                                   \\
        &= Q (D_x D^{-1} D \overline{D^{-1}} D_{\overline{x}}) Q^\ast                                   \\
        &= (Q \overline{D^{-1}} D_{\overline{x}} Q^\ast)(Q \overline{D^{-1}} D_{\overline{x}} Q^\ast)   \\
        &= M_x M_x^\ast
    \end{align*}
    If $\Im Q = 0$, then $Q$ can be viewed as a real matrix via the mapping $Q \longmapsto \Re Q$ and, in this case, $M_x$ is symmetric since $(QD_x D^{-1} Q^\top)^\top = Q D^{-1} D_x Q^\top = QD_x D^{-1} Q^\top$. Thus, \thref{ineqredux} yields novel sufficient conditions for the normal NIEP when $\Im Q \ne 0$ and the symmetric NIEP when $\Im Q = 0$.
\end{remark}

\section{Volume of the Projected Spectratope of a Real Character Table}

\begin{lemma}
\thlabel{scalingConeLemma}
    If $S$ is ideal and $v > 0$, then $D_v S$ is ideal.
\end{lemma}

\begin{proof}
    Recall that $\mathcal{C}(S) = \mathcal{C}(D_v S)$ and if $\alpha_1,\ldots,\alpha_n$ are positive scalars, then $\conv{(v_1,\ldots,v_n)} = \conv{(\alpha_1 v_1,\ldots, \alpha_n v_n)}$.
\end{proof}

If $v \coloneqq Q e_1$, then, since $Q$ is ideal, it follows that $D_{v^{-1}} Q$ is ideal and $(D_{v^{-1}} Q) e_1 = e$. Recall, by \thref{P(S)contain}, that $\mathcal{P}(Q) \subseteq \mathcal{P}(D_{v^{-1}}Q)$.

If $v_0,v_1,\ldots,v_n \in \mathbb{R}^n$ are \emph{affinely independent}, i.e., the vectors $v_1 - v_0,\ldots,v_n-v_0$ are linearly independent, then $S := \conv(v_0,v_1,\ldots,v_n)$ is called an \emph{$n$-simplex}. It is known \cite{s1966} that the volume $V$ of $S$ is given by 
\begin{equation}
V  
= \dfrac{1}{n!} \left\vert \det \left(\begin{bmatrix} v_0 & v_1 & \cdots & v_n \\ 1 & 1 & \cdots & 1 \end{bmatrix}\right) \right\vert
= \dfrac{1}{n!} \left\vert \det \left(\begin{bmatrix} 1 & v_0^\top \\ 1 & v_1^\top \\ \vdots & \vdots \\ 1 & v_n^\top \end{bmatrix}\right) \right\vert
\label{vol}
\end{equation} 

For $k \in \bracket{n}$, let $P_k$ the matrix obtained by deleting the $k\textsuperscript{th}$-row of $I_n$ and let $\Pi_k: \mathbb{F}^n \longrightarrow \mathbb{F}^{n-1}$ be the projection map defined by $\Pi_k(x) = P_k x$.

In what follows, it is assumed that $\Im Q = 0$ and $Q$ is viewed as the real matrix $\Re Q$.

\begin{theorem}
    \thlabel{volumeSpectratope}
    The volume of $\Pi_1 (\mathcal{P} (D_{v^{-1}}Q))$ is given by 
    \begin{equation}
        V = \dfrac{\sqrt{\prod_{k=1}^n |C_G(g_k)| }}{(n-1)!\prod_{k=1}^n \dim(\rho_k)}. \label{volumeformula}
    \end{equation}
\end{theorem}

\begin{proof}
    Note that 
    $$D_{v^{-1}} = \text{diag}\left(\dfrac{1}{\dim(\rho_1)},\cdots, \dfrac{1}{\dim(\rho_n)}\right)$$
    and 
    $$\det \left(D_{v^{-1}}\right) = \dfrac{1}{\prod_{k=1}^n \dim \left(\rho_k\right)}.$$
    Since $Q^\top Q$ is a diagonal matrix with $k\textsuperscript{th}$ diagonal entry $\vert C_G(g_k)\vert$, it follows that 
    $$\det Q = \sqrt{\prod_{k=1}^n \vert C_G(g_k)\vert}.$$ 
    
    Applying \eqref{vol} and properties of the determinant, 
    \begin{align*}
        V 
        &= \dfrac{1}{(n-1)!}\left\vert \det \left(D_{v^{-1}} Q\right) \right\vert \\
        &= \dfrac{1}{(n-1)!} \vert \det \left(D_{v^{-1}} \right) \vert \vert \det(Q) \vert \\
        &= \dfrac{\sqrt{\prod_{k=1}^n |C_G(g_k)| }}{(n-1)!\prod_{k=1}^n \dim(\rho_k)}. \qedhere 
    \end{align*}
\end{proof}

\section{Examples}

If $\Lambda = \{\lambda_1 \coloneqq 1,\lambda_1,\ldots,\lambda_n\}$ is a list of real numbers and $1 \leqslant n \leqslant 3$, then $\Lambda$ is realizable if and only if conditions \ref{sprad} and \ref{trnn} are satisfied \cite{ll1978-79}, i.e., if and only if 
\begin{align}
    1 &\geqslant \vert \lambda_k \vert          \label{sprad2}  \\
    1 + \sum_{k=1}^n \lambda_k &\geqslant 0.     \label{trnn2}
\end{align}
Conditions \ref{sprad2} and \ref{trnn2} define a polytope called the \emph{trace nonnegative polytope}, whose volume is given by 
\begin{equation*}
2^n \left[1 - \frac{1}{n!}\sum_{k=0}^{\lfloor \frac{n-1}{2} \rfloor} (-1)^k{n \choose k} \left(\frac{n-1}{2}-k \right)^n \right].
\end{equation*}
(Taylor and Paparella \cite[Corollary 2.2]{tp2019}).

\begin{example}
    The character table for $\mathbb{Z}_2$ is the Walsh matrix
    \[ H_2 = 
    \begin{bmatrix}
        1 & 1 \\
        1 & -1
    \end{bmatrix} \]
    and the cone $\mathcal{C}(H_2)$ yields all possible spectra since 
    \[ H_2 D_x H_2^{-1} = \frac{1}{2} H_2 D_x H_2^\top =
    \frac{1}{2}
    \begin{bmatrix}
        x_1 + x_2 & x_1 - x_2 \\
        x_1 - x_2 & x_1 + x_2
    \end{bmatrix}.
    \]
\end{example}

\begin{example}
The character table for $\mathsf{Sym}(3)$ is 
    $$
    Q = \begin{bmatrix}
        1 & 1 & 1\\
        1 & -1 & 1\\
        2 & 0 & -1
    \end{bmatrix}
    $$
and
    $$
    M_x (Q) = \frac{1}{6} \begin{bmatrix}
        x_1 + 3x_2 + 2x_3 & x_1 - 3x_2 + 2x_3 & 2(x_1 - x_3)\\
        x_1 - 3x_2 + 2x_3 & x_1 + 3x_2 + 2x_3 & 2(x_1 - x_3)\\
        2(x_1 - x_3) & 2(x_1 - x_3) & 2(2x_1 + x_3)
    \end{bmatrix}.
    $$
By \thref{ineqredux}, the half-spaces that determine the polyhedral cone $\mathcal{C}(Q)$ are $x_1 + 3x_2 + 2x_3 \geqslant 0$, $x_1 - 3x_2 + 2x_3 \geqslant 0$, and $2(x_1 - x_3) \geqslant 0$. Indeed, the inequality $2(2x_1 + x_3) \geqslant 0$ is redundant since 
    \begin{align*}
            & [M_x (Q)]_{3 3} \\
        =   &\frac{1}{6} \sum_{\ell = 1}^3 \langle \chi_{\rho_3 \otimes \rho_3^*}, \chi_\ell \rangle \left( \sum_{k = 1}^3 |\cl(g_k)| \chi_\ell(g_k) x_k \right)\\
        =   &\frac{1}{6} \left( \frac{4 + 2}{6}(x_1 + 3x_2 + 2x_3) + \frac{4 + 2}{6} (x_1 - 3x_2 + 2x_3) + \frac{8 - 2}{6} (2(x_1 - x_3)) \right)\\
        =   &\frac{1}{6} (2(2x_1 + x_3))
    \end{align*} 
which is nonnegative whenever the other three inequalities are nonnegative. 

If $v=\begin{bmatrix} 1 & 1 & 2 \end{bmatrix}^\top$, then 
    \[ D_{v^{-1}}Q = \begin{bmatrix}
        1 & 1 & 1\\
        1 & -1 & 1\\
        1 & 0 & -0.5
    \end{bmatrix} \]
Applying  formula \ref{volumeformula}, the area of $\Pi_1 (\mathcal{P} (D_{v^{-1}}Q))$ is given by
    \begin{align*} 
        V &= \frac{\sqrt{6 \cdot 3 \cdot 2}}{(3 - 1)! \cdot 1 \cdot 1 \cdot 2} = \frac{3}{2}.
    \end{align*}
The area of the \emph{trace nonnegative polytope} is ${7}/{2}$ so that $\Pi_1 (\mathcal{P} (D_{v^{-1}}Q))$ occupies ${3}/{7}$ of the feasible region (see Figure \ref{projex}). 
    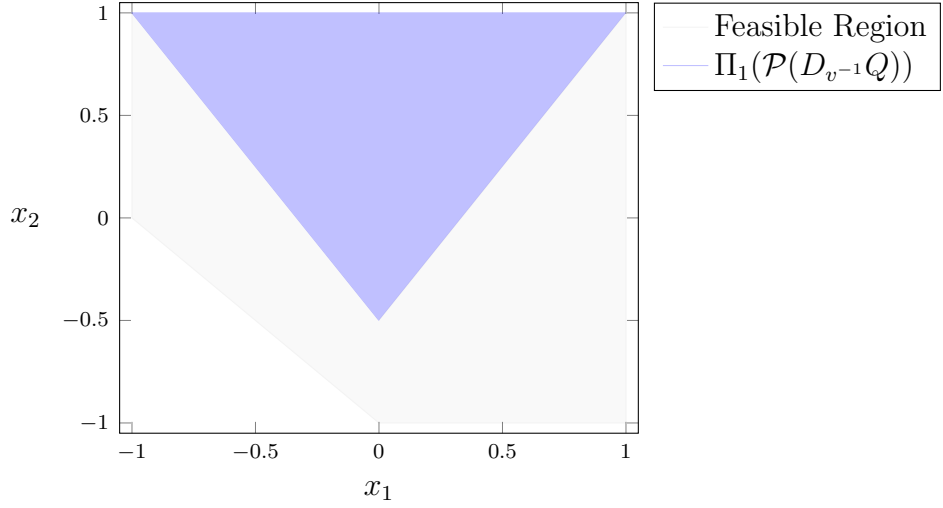
\begin{figure}[H]
    \centering
    \begin{tikzpicture}
    \begin{axis}[
    xticklabel style={font=\scriptsize},
    yticklabel style={font=\scriptsize},
    xmin=-1.05,
    xmax=1.05,
    ymin=-1.05,
    ymax=1.05,
    y label style={rotate=-90},
    xlabel=$x_1$,
    ylabel=$x_2$,
    legend entries={Feasible Region,$\Pi_1 (\mathcal{P} (D_{v^{-1}}Q))$},
    legend style={nodes=right},
    legend pos= outer north east]
    
    \addplot[fill=Gray,opacity=0.05] coordinates{(-1,1)(-1,0)(0,-1)(1,-1)(1,1)(0,-.5)}; 
    \addplot[Blue,fill=Blue,opacity=0.25] coordinates{(1,1)(-1,1)(0,-.5)}; 
    
    \end{axis}
    \end{tikzpicture}
    \caption{Feasible Region and $\Pi_1 (\mathcal{P} (D_{v^{-1}}Q))$.}
    \label{projex}
    \end{figure}
\end{example}

\begin{example}
The character table of $\mathbb{Z}_2\oplus\mathbb{Z}_2$ is the Walsh matrix
$$H_{4} = 
\begin{bmatrix}
    1 & 1 & 1 & 1 \\ 1 & -1 & 1 & -1 \\ 1 & 1 & -1 & -1 \\ 1 & -1 & -1 & 1. 
\end{bmatrix}$$
If $x \in \mathbb{R}^4$, then, following \thref{ineqredux}, $M_x(H_4) \geqslant 0$ if and only if
    \begin{align*}
    {x_1 + x_2 + x_3 + x_4} \geqslant 0   \\
    {x_1 - x_2 + x_3 - x_4} \geqslant 0  \\
    {x_1 + x_2 - x_3 - x_4} \geqslant 0  \\
    {x_1 - x_2 - x_3 + x_4} \geqslant 0
    \end{align*} 
As $\mathbb{Z}_2 \oplus \mathbb{Z}_2$ is Abelian, $\vert C_{\mathbb{Z}_2 \oplus \mathbb{Z}_2}(g)\vert = 4,\ \forall g \in \mathbb{Z}_2 \oplus \mathbb{Z}_2.$ 
Consequently, applying formula \ref{volumeformula} yields
    \begin{align*} 
        V &= \frac{\sqrt{4^4}}{(4 - 1)! \cdot 1^4} = \frac{8}{3}.
    \end{align*}
The volume of the \emph{trace nonnegative polytope} is ${20}/{3}$ so that $\Pi_1 (\mathcal{P} (H_4))$ occupies ${2}/{5}$ of the feasible region (see Figure \ref{hadspectwo}).

    \begin{figure}[H]
    \centering
    \begin{tikzpicture}
    \begin{axis}
    [view={-35}{20},
    xmin=-1.05,
    xmax=1.05,
    ymin=-1.05,
    ymax=1.05,
    zmin=-1.05,
    zmax=1.05,
    z label style={rotate=-90},
    xlabel=$x_1$,
    ylabel=$x_2$,
    zlabel=$x_3$,    
    legend entries={Feasible Region,$\Pi_1 (\mathcal{P} (H_4))$},
    legend style={nodes=right},
    legend pos= outer north east]

        \addplot3[fill=Gray,opacity=.05] coordinates{(-1,-1,1) (1,-1,1)(1,1,1)(-1,1,1)};
        \addplot3[forget plot,fill=Gray,opacity=.05] coordinates{(-1,-1,1)(-1,1,1)(-1,1,-1)};
        \addplot3[forget plot,fill=Gray,opacity=.05] coordinates{(-1,-1,1)(1,-1,1)(1,-1,-1)};
        
        \addplot3[Blue,fill=Blue,opacity=.25] coordinates{(-1,-1,1) (-1,1,-1) (1,-1,-1) (-1,-1,1)};
        \addplot3[forget plot,Blue,fill=Blue,opacity=.25] coordinates{(1,1,1) (-1,-1,1) (1,-1,-1) (1,1,1)};
        \addplot3[forget plot,Blue,fill=Blue,opacity=.25] coordinates{(1,1,1) (-1,-1,1) (-1,1,-1) (1,1,1)};
    
    \end{axis}
    \end{tikzpicture}
    \caption{Feasible Region and $\Pi_1 (\mathcal{P} (H_4))$.}
    \label{hadspectwo}
    \end{figure}
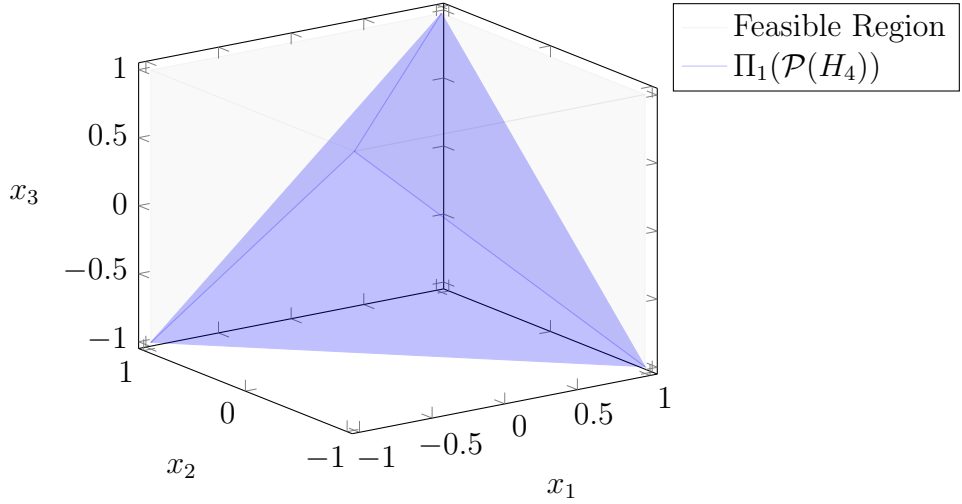
Although this polytope only occupies less than half of the realizable spectra, it is known that all possible spectra can be realized via the similarities $H_4$ and $H_2 \oplus H_2$ \cite[pp.~290--291]{jp2016}. 
\end{example}

\section{Concluding Remarks \& Further Inquiry}

In what follows, it is assumed, without loss of generality, that $S$ is an ideal Perron similarity normalized such that $S e_1 = e$ and $\vert{s_{ij}}\vert \leqslant 1$ [\S\S 4.3]\cite{jp2025}. In such a case, it is known that $\mathcal{C}(S) = \mathcal{C}_r(S)$ if and only if $\mathcal{P}(S) = \mathcal{P}_r(S)$, where $\mathcal{P}_r(S)$ denotes the convex hull of the rows of $S$ \cite[Theorem 4.22]{jp2025}. 

For $x \in \mathbb{C}^n$, denote by $\Lambda(x)$ the list $\{ x_1, \dots, x_n\}$ and for every natural number $n$, let 
\[ \mathbb{L}^n \coloneqq \{ x \in \mathbb{C}^n \mid \Lambda(x) = \spec(A),~A \in \mathsf{M}_{n}\left( \mathbb{R} \right),~A\geqslant 0 \} \]
and 
\[ \mathbb{SL}^n \coloneqq \{ x \in \mathbb{C}^n \mid \Lambda(x) = \spec(A),~A \in \mathsf{M}_{n}\left( \mathbb{R} \right),~A~\text{stochastic} \}. \]
The set $\mathbb{L}^n$ is a cone that contains $\mathbb{SL}^n$ and a characterization of either set constitutes a solution to the NIEP. It is known that $\mathbb{SL}^n$ is compact and star-shaped at $e$ \cite[Theorems 3.9 and 3.10]{jp2025}. 

If $x \in \mathbb{SL}^n$, then $1 = \rho\left(\Lambda(x)\right) \in \Lambda(x)$ and if $P$ is a permutation matrix, then \( \Lambda(Px) = \Lambda(x) \). In light of these two facts, there is no loss in generality in assuming that $x_1 = 1$. It is known that $\Pi_1(\mathbb{SL}^n)$ is star-shaped at the origin \cite[Theorem 3.12]{jp2025}.

If $x = \begin{bmatrix} 1 & x_2 & \cdots & x_n \end{bmatrix}^\top \in \mathbb{SL}^n$, then $x$ is called \emph{extremal (in $\mathbb{SL}^n$)} if $\alpha \Pi_1(x) \not \in \Pi_1(\mathbb{SL}^n)$, $\forall\alpha > 1$. The set of extremal points in $\mathbb{SL}^n$ is denoted by $\mathbb{E}_n$. Johnson and Paparella showed that $\mathbb{E}_n \subseteq \partial \mathbb{SL}^n$ and conjectured the reverse containment. Thus, identifying extremal points is of great interest in the NIEP.

A Perron similarity $S$ is called \emph{extremal} if $\mathcal{P}(S)$ contains an extremal point other than $e$.

If $G$ is a finite Abelian group, then every irreducible character is one-dimensional and the number of irreducible characters
is equal to the order of the group \cite[Theorem 10.5.2(a)]{a1991}. If $\chi_\rho$ is a one-dimensional character and $g$ is an element of $G$, then $\chi_\rho(g)$ is a power of the primitive root of unity $\omega_{\vert{g}\vert}$ \cite[p.\ 303]{a1991}. Furthermore, since there are prime numbers $p_1,\dots,p_k$ (not necessarily distinct) and positive integers $n_1,\ldots,n_k$ such that 
$$G \cong \mathbb{Z}_{p_1^{n_1}}\oplus \cdots \oplus\mathbb{Z}_{p_k^{n_k}}$$
and the discrete Fourier transform matrix $F_n$ is the character table of $\mathbb{Z}_n$, it follows that
$$F_{p_1}^{\otimes n_1} \otimes \cdots \otimes F_{p_1}^{\otimes n_1}$$
is the character table of $G$. If $G$ has even order, then the character table is of the form
$$H_{2^{n_0}} \otimes F_{p_1}^{\otimes n_1} \otimes \cdots \otimes F_{p_1}^{\otimes n_1}.$$

Karpelevi{\v{c}} \cite{k1951} (and previously, Romanovsky \cite{r1936}) proved that the so-called \emph{Karpelevi{\v{c}} region} 
$$\Theta_n \coloneqq \{ \lambda \in \mathbb{C} \mid \lambda \in \spec A,\ A\geqslant 0, Ae = e\}$$ 
intersects the unit-circle $\{ z \in \mathbb{C} \mid |z| = 1\}$ at the points $\{ \omega_q^p \mid p/q \in \mathcal{F}_n \}$, where $\mathcal{F}_n \coloneqq  \{ p/q \mid 0 \leqslant p \leqslant q \leqslant n,~\gcd(p,q)=1 \}$ denotes the set of \emph{Farey fractions}.

Thus, the character table of a finite Abelian group is \emph{totally extremal} because every entry is extremal in $\Theta_n$ (see Figure \ref{hadspectwo}; see Johnson and Paparella [\S\S 8.1]\cite{jp2025} for geometrical representations in the four-dimensional complex case).  

In view of the above, we offer the following for further inquiry.

\begin{conjecture}
    If $S$ is a normalized ideal Perron similarity and $S$ is totally extremal, then $S$ is the character table of a finite Abelian group.     
\end{conjecture}

\bibliographystyle{abbrv}
\bibliography{references}

\begin{thebibliography}{10}

\bibitem{a1991}
M.~Artin.
\newblock {\em Algebra}.
\newblock Prentice Hall, Inc., Englewood Cliffs, NJ, 1991.

\bibitem{c1998}
M.~T. Chu.
\newblock Inverse eigenvalue problems.
\newblock {\em SIAM Rev.}, 40(1):1--39, 1998.

\bibitem{dppt2022}
J.~M. Dockter, P.~Paparella, R.~L. Perry, and J.~D. Ta.
\newblock Kronecker products of {P}erron similarities.
\newblock {\em Electron. J. Linear Algebra}, 38:114--122, 2022.

\bibitem{df2004}
D.~S. Dummit and R.~M. Foote.
\newblock {\em Abstract algebra}.
\newblock John Wiley \& Sons, Inc., Hoboken, NJ, third edition, 2004.

\bibitem{fh1991}
W.~Fulton and J.~Harris.
\newblock {\em Representation theory}, volume 129 of {\em Graduate Texts in Mathematics}.
\newblock Springer-Verlag, New York, 1991.
\newblock A first course, Readings in Mathematics.

\bibitem{hj2013}
R.~A. Horn and C.~R. Johnson.
\newblock {\em Matrix analysis}.
\newblock Cambridge University Press, Cambridge, second edition, 2013.

\bibitem{j1981}
C.~R. Johnson.
\newblock Row stochastic matrices similar to doubly stochastic matrices.
\newblock {\em Linear and Multilinear Algebra}, 10(2):113--130, 1981.

\bibitem{jmpp2018}
C.~R. Johnson, C.~Mariju\'an, P.~Paparella, and M.~Pisonero.
\newblock The {NIEP}.
\newblock In {\em Operator theory, operator algebras, and matrix theory}, volume 267 of {\em Oper. Theory Adv. Appl.}, pages 199--220. Birkh\"auser/Springer, Cham, 2018.

\bibitem{jp2016}
C.~R. Johnson and P.~Paparella.
\newblock Perron spectratopes and the real nonnegative inverse eigenvalue problem.
\newblock {\em Linear Algebra Appl.}, 493:281--300, 2016.

\bibitem{jp2017}
C.~R. Johnson and P.~Paparella.
\newblock Row cones, {P}erron similarities, and nonnegative spectra.
\newblock {\em Linear Multilinear Algebra}, 65(10):2124--2130, 2017.

\bibitem{jp2025}
C.~R. Johnson and P.~Paparella.
\newblock Perron similarities and the nonnegative inverse eigenvalue problem.
\newblock {\em Trans. Amer. Math. Soc.}, 378(12):8361--8389, 2025.

\bibitem{k1951}
F.~I. Karpelevi\v{c}.
\newblock On the characteristic roots of matrices with nonnegative elements.
\newblock {\em Izvestiya Akad. Nauk SSSR. Ser. Mat.}, pages 361--383, 1951.

\bibitem{ll1978-79}
R.~Loewy and D.~London.
\newblock A note on an inverse problem for nonnegative matrices.
\newblock {\em Linear and Multilinear Algebra}, 6(1):83--90, 1978/79.

\bibitem{r1936}
V.~Romanovsky.
\newblock Recherches sur les cha\^{i}nes de {M}arkoff.
\newblock {\em Acta Math.}, 66(1):147--251, 1936.
\newblock Premier M\'{e}moire.

\bibitem{s1977}
J.-P. Serre.
\newblock {\em Linear representations of finite groups}, volume Vol. 42 of {\em Graduate Texts in Mathematics}.
\newblock Springer-Verlag, New York-Heidelberg, french edition, 1977.

\bibitem{s1966}
P.~Stein.
\newblock Classroom {N}otes: {A} {N}ote on the {V}olume of a {S}implex.
\newblock {\em Amer. Math. Monthly}, 73(3):299--301, 1966.

\bibitem{tp2019}
G.~K. Taylor and P.~Paparella.
\newblock The volume of the trace nonnegative polytope via the {I}rwin--{H}all distribution.
\newblock {\em The Minnesota Journal of Undergraduate Mathematics}, 4(1), 2019.

\end{thebibliography}

\end{document}